\documentclass{amsart}
\usepackage{amssymb}
\usepackage{amscd}
\usepackage[latin1]{inputenc}
\usepackage{amsmath}
\usepackage{amsfonts}
\usepackage{latexsym}
\usepackage{verbatim}
\usepackage{graphicx}
\usepackage{epsfig}
\newtheorem{theorem}{Theorem}[section]

\newtheorem{lemma}[theorem]{Lemma}
\newtheorem{question}[theorem]{Question}

\theoremstyle{definition}
\newtheorem{definition}[theorem]{Definition}

\newtheorem{remark}[theorem]{Remark}

\newcommand{\zz}{\mathbb{Z}}

\begin{document}
\title[Intrinsic ergodicity and weakenings of specification]{On intrinsic ergodicity and weakenings of the specification property}

\begin{abstract}
Since seminal work of Bowen (\cite{bowen}), it has been known that the specification property implies various useful properties about
a topological dynamical system, among them uniqueness of the measure of maximal entropy (often referred to as intrinsic ergodicity). 
Weakenings of the specification property called almost weak specification and almost specification have been defined and profitably 
applied in various works such as \cite{kwietniaketal}, \cite{marcusmonatshefte}, and \cite{thompson}.

However, it has been an open question (see p. 798 of \cite{CT}) whether either or both of these properties imply intrinsic
ergodicity. We answer this question negatively by exhibiting examples of subshifts with multiple measures
of maximal entropy with disjoint support which have almost weak specification with any gap function $f(n) = O(\ln n)$ or
almost specification with any mistake function $g(n) \geq 4$. We also show some results in the opposite direction, showing that subshifts with
almost weak specification with gap function $f(n) = o(\ln n)$ or almost specification with mistake function $g(n) = 1$
cannot have multiple measures of maximal entropy with disjoint support. 

\end{abstract}

\date{}
\author{Ronnie Pavlov}
\address{Ronnie Pavlov\\
Department of Mathematics\\
University of Denver\\
2280 S. Vine St.\\
Denver, CO 80208}
\email{rpavlov@du.edu}
\urladdr{www.math.du.edu/$\sim$rpavlov/}
\thanks{}
\renewcommand{\subjclassname}{MSC 2010}
\maketitle



\section{Introduction}\label{intro}

Entropy is one of the most well-studied invariants in the field of dynamical systems. Entropy can be defined both for measure-theoretic dynamical systems (given by a probability space $(X,\mu)$ and $T$ a $\mu$-preserving self-map of $X$) and for topological dynamical systems (given by a compact topological space $X$ and $T$ a continuous self-map of of $X$). A relation between the two notions of entropy is given by the celebrated Variational Principle, which states that for any topological dynamical system $(X,T)$, the topological entropy is the supremum over all measure-theoretic entropies for Borel measures $\mu$ on $X$ which are preserved by $T$. For this reason, such a measure on $X$ whose entropy achieves this supremum is called a \textbf{measure of maximal entropy}.

It is well-known that expansive topological dynamical systems always have at least one measure of maximal entropy, and symbolic systems/subshifts are some of the best-known examples of expansive topological dynamical systems. A system is called \textbf{intrinsically ergodic} when this measure is unique, and establishing intrinsic ergodicity is a central problem in both ergodic theory and topological dynamics. A common way of proving intrinsic ergodicity is via strong enough specification properties. Specification properties involve combining segments of orbits into a new orbit in various ways; in subshifts, these orbit segments can be represented by words occurring in points of the subshift. We consider two weakenings of the classical specification property in the symbolic setting; the first (called \textbf{almost weak specification}) allows one to combine arbitrary words in the language into a new word in the language if gaps which are ``small'' in comparison to the combined words are placed in between (controlled by a function $f(n) = o(n)$), and the second (called \textbf{almost specification}) allows one to concatenate arbitrary words in the language into a new word in the language if a ``small'' number of letters are allowed to change in each word (controlled by a function $g(n) = o(n)$). See Section~\ref{defs} for formal definitions. 

In \cite{CT}, (p. 798), a question was posed as to whether almost weak specification and/or almost specification implies intrinsic ergodicity. We answer this question negatively, by exhibiting two different examples of subshifts which have almost weak specification and almost specification, respectively, and yet have multiple measures of maximal entropy. 

\begin{theorem}\label{mainex}
For any positive nondecreasing function $f(n)$ with $\liminf_{n \rightarrow \infty} \frac{f(n)}{\ln n} > 0$, there exists a subshift with almost weak specification with gap function $f(n)$ with exactly two ergodic measures of maximal entropy, whose supports are disjoint.
\end{theorem}

\begin{theorem}\label{mainex2}
There exists a subshift with almost specification with mistake function $g(n) = 4$ with two ergodic measures of maximal entropy, whose supports are disjoint.
\end{theorem}

It is particularly surprising that even boundedness of the mistake function for systems with almost specification does not imply uniqueness of the measure of maximal entropy. We also prove some results in the opposite direction, proving that if $f(n)$ and/or $g(n)$ grows extremely slowly, then the subshift cannot have two measures of maximal entropy with disjoint supports. This result in some sense precludes a ``strong nonuniqueness'' of the measures of maximal entropy.

\begin{theorem}\label{mainthm}
If a subshift has almost weak specification with gap function $f(n)$ where $\liminf_{n \rightarrow \infty} \frac{f(n)}{\ln n} = 0$, then it cannot have two measures of maximal entropy with disjoint support.
\end{theorem}

\begin{theorem}\label{mainthm2} 
If a subshift has almost specification with mistake function $g(n) = 1$, then it cannot have two measures of maximal entropy with disjoint support.
\end{theorem}

We have then completely answered the question of whether almost weak specification for a particular gap function $f(n)$ can coexist with multiple measures of maximal entropy with disjoint supports, and leave open only the case $1 \leq g(n) \leq 4$ for the corresponding question for almost specification. 
This suggests that for both almost weak specification and almost specification, there is a ``phase transition'' 
significantly below $n$ for the relevant gap or mistake function for how the property influences the measures of maximal entropy.
It is still plausible that an extremely slow growth rate for $f(n)$ and/or $g(n)$ may imply uniqueness of the measure of maximal entropy, but we do not know whether this is true or not.

\begin{question}\label{mainq}
Do there exist positive functions $F(n)$ and/or $G(n)$ so that almost weak specification with a gap function $f(n) \leq F(n)$ and/or almost specification with a mistake function $g(n) \leq G(n)$ forces intrinsic ergodicity?
\end{question}

We quickly summarize the structure of the paper. Section~\ref{defs} gives formal definitions of all relevant concepts, Section~\ref{AWS} gives proofs of our results concerning almost weak specification (namely Theorems~\ref{mainex} and~\ref{mainthm}), and Section~\ref{AS} gives proofs of our results concerning almost specification (namely Theorems~\ref{mainex2} and~\ref{mainthm2}).

\begin{remark}
We would like to point out that the question of whether almost weak/almost specification implies intrinsic ergodicity has been independently answered negatively by Kwietniak, Oprocha, and Rams (\cite{KOR}). They also prove other interesting results about implications of and relations between almost weak specification, almost specification, and the so-called Climenhaga-Thompson decomposition from \cite{CT}.
\end{remark}

\section{Definitions and preliminaries}\label{defs}

\begin{definition}
For any finite alphabet $A$, the \textbf{full shift} over $A$ is the set $A^{\zz} = \{\ldots x_{-1} x_0 x_1 \ldots \ : \ x_i \in A\}$, which is viewed as a compact topological space with the (discrete) product topology.
\end{definition}

\begin{definition}
A \textbf{word} over $A$ is a member of $A^{\{i,i+1,\ldots,j\}}$ for some $i<j$, whose \textbf{length} $j-i+1$ is denoted by $|w|$. The set $\bigcup_{i,j \in \zz, i<j} A^{\{i,i+1,\ldots,j\}}$ of all words over $A$ is denoted by $A^*$. For any $n$, we use $A^n$ to denote the set $A^{\{1,\ldots,n\}}$.
\end{definition}

\begin{definition}
The \textbf{shift action}, denoted by $\{\sigma^n\}_{t \in \zz}$, is the $\zz$-action on a full shift $A^{\zz}$ defined by $(\sigma^n x)_m = x_{m+n}$ for $m,n \in \zz$. 
\end{definition}

\begin{definition}
A \textbf{subshift} is a closed subset of a full shift $A^{\zz}$ which is invariant under the shift action, which is a compact space with the induced topology from $A^{\zz}$.
\end{definition}

The single shift $\sigma := \sigma^1$ is an automorphism on any subshift, and so for any subshift $X$, $(X,\sigma)$ is a topological dynamical system. An alternate definition for a subshift is in terms of a list of forbidden words; for any set $\mathcal{F} \subset A^*$, one can define the set $X(\mathcal{F}) := \{x \in A^{\zz} \ : \ x_i x_{i+1} \ldots x_j \notin \mathcal{F} \ \forall i,j \in \zz, i<j\}$. It is well known that any $X(\mathcal{F})$ is a subshift, and all subshifts are representable in this way.

\begin{definition}
The \textbf{language} of a subshift $X$, denoted by $\mathcal{L}(X)$, is the set of all words which appear in points of $X$. For any $n \in \zz$, $\mathcal{L}_n(X) := 
\mathcal{L}(X) \cap A^n$, the set of words in the language of $X$ with length $n$. 
\end{definition}

In the previous definition, we dealt only with words from $A^n$ rather than $A^{\{i,\ldots,j\}}$ for arbitrary $i < j$; this is because any word in $A^{\{i,\ldots,j\}}$ can clearly be thought of as a word in $A^{j-i+1}$ by simply shifting it. We will generally consider two words to be the same if they are shifts of each other. 

\begin{definition}
For any subshift and word $w \in \mathcal{L}_n(X)$, the \textbf{cylinder set} $[w]$ is the set of all $x \in X$ with $x_1 x_2 \ldots x_n = w$.  
\end{definition}

\begin{definition}
For any subshift $X \subset A^{\zz}$ and any $k \in \mathbb{N}$, the \textbf{$k$th higher-power shift} associated to $X$, denoted $X^k$, is a subshift with alphabet $\mathcal{L}_k(X)$ defined by the following rule: $y \in (\mathcal{L}_k(X))^{\mathbb{Z}}$ is an element of $X^k$ if and only if the point $x$ defined by concatenating the ``letters'' of $y$ is in $X$. 
(Formally, $\forall n \in \zz$, the $n$th letter of $x$ is defined to be the $(n \pmod k)$th letter of $y_{\lfloor n/k \rfloor}$.)
\end{definition}

It is well-known that the dynamical systems $(X^k, \sigma)$ and $(X, \sigma^k)$ are topologically isomorphic.

\begin{definition}\label{topent}
The \textbf{topological entropy} of a subshift $X$ is
\[
h(X) := \lim_{n \rightarrow \infty} \frac{1}{n} \ln |\mathcal{L}_n(X)|.
\]
\end{definition}

We also need some definitions from measure-theoretic dynamics; all measures considered in this paper will be Borel probability measures on a full shift $A^{\mathbb{Z}}$.

\begin{definition}
A measure $\mu$ on $A^{\mathbb{Z}}$ is {\bf ergodic} if any measurable set $C$ which is shift-invariant, meaning $\mu(C \triangle \sigma C) = 0$, has measure $0$ or $1$. 
\end{definition}

Not all measures are ergodic, but a well-known result called the ergodic decomposition shows that any non-ergodic measure can be written as an ``average'' (formally, an integral) of ergodic measures; see 
Chapter 6 of \cite{walters} for more information. One application of ergodic measures comes from Birkhoff's pointwise ergodic theorem, stated here only for the case of ergodic $\mu$ on a full shift $A^{\zz}$.

\begin{theorem}{\rm (Birkhoff's pointwise ergodic theorem)}\label{birkhoff}
For any ergodic measure $\mu$ on a subshift $X$ and any $f \in L^1(A^{\zz},\mu)$,
\[
\lim_{n \rightarrow \infty} \frac{1}{2n+1} \sum_{i=-n}^n f(\sigma^i x) \underset{\mu{\rm -a.e.}}{\rightarrow} \int f \ d\mu.
\]
\end{theorem}

\begin{definition}\label{measent}
For any measure $\mu$ on a full shift $A^{\zz}$, the \textbf{measure-theoretic entropy} of $\mu$ is
\[
h(\mu) := \lim_{n \rightarrow \infty} \frac{-1}{n}  \sum_{w \in A^n} \mu([w]) \ln \mu([w]),
\]
where terms with $\mu([w]) = 0$ are omitted from the sum.
\end{definition}

In Definitions~\ref{topent} and \ref{measent}, a standard subadditivity argument shows that the limits can be replaced by infimums; i.e. for any $n$, 
$h(X) \leq \frac{1}{n} \ln |\mathcal{L}_n(X)|$ and $h(\mu) \leq \frac{-1}{n} \sum_{w \in A^n} \mu([w]) \ln \mu([w])$. This implies the following fact.

\begin{lemma}\label{supportcount}
For any measure $\mu$ on a shift space $X$,
\[
|\{w \in \mathcal{L}_n(X) \ : \ \mu([w]) > 0\}| \geq e^{nh(\mu)}.
\]
\end{lemma}

\begin{proof}

Choose any such $X$, $\mu$, and $n$, and denote the set in the lemma by $S$. It is easily checked that for any probability vector $(x_1, \ldots, x_n)$, $-\sum_{i=1}^n x_i \ln x_i \leq \ln n$, and equality is achieved if and only if all $x_i$ are equal to $\frac{1}{n}$ (see \cite{walters}, Corollary 4.2.1 for a proof). Therefore, 
\[
h(\mu) \leq \frac{-1}{n} \sum_{w \in A^n} \mu([w]) \ln \mu([w]) \leq \frac{\ln |S|}{n}.
\]

This implies that $|S| \geq e^{nh(\mu)}$.

\end{proof}

We also need the following fact, whose proof can be found as Lemma 4.8 in \cite{pavlovperturb}. 

\begin{lemma}\label{genericcount}
For any subshift $X$, any ergodic measure $\mu$ on $X$, any finite
set of words $w_i \in \mathcal{L}_{n_i}(X)$ for $1 \leq i \leq j$, any $k \in \mathbb{N}$, and any
$\epsilon > 0$, define the set $C_{k,\epsilon,w_1,\ldots,w_j}(X)$ to be the set of all $w \in \mathcal{L}_k(X)$
which have between $k(\mu([w_i]) - \epsilon)$ and $k(\mu([w_i]) + \epsilon)$ occurrences of $w_i$ for each $i$.
Then,
\[
\liminf_{k \rightarrow \infty} \frac{\ln |C_{k,\epsilon,w_1,\ldots,w_j}(X)|}{k} \geq h(\mu).
\] 
\end{lemma}

\begin{definition}
For any subshift $X$, a \textbf{measure of maximal entropy} on $X$ is a measure $\mu$ with support contained in $X$ for which $h(\mu) = h(X)$.
\end{definition}


As noted in the introduction, any subshift has at least one measure of maximal entropy. In fact, the ergodic decomposition, along with the fact that the entropy map $\mu \mapsto h(\mu)$ is affine (See Theorems 6.10 and 8.1 in \cite{walters}), implies that the extreme points of the simplex of measures of maximal entropy are precisely the ergodic measures of maximal entropy, and so in particular, any subshift also has an ergodic measure of maximal entropy. 

We will need to make use of the following fact about the full shift.

\begin{lemma}\label{fsunique}
For any alphabet $A$, the full shift $A^{\mathbb{Z}}$ has a unique measure of maximal entropy, namely the measure $\mu$ with $\mu([w]) = |A|^{-n}$ for all $n$ and $w \in A^n$.
\end{lemma}

\begin{proof}
The topological entropy of the full shift $A^{\mathbb{Z}}$ is
\[
h(A^{\mathbb{Z}}) = \lim_{n \rightarrow \infty} \frac{1}{n} \ln|\mathcal{L}_n(A^{\mathbb{Z}})| = \lim_{n \rightarrow \infty} \frac{1}{n} \ln (|A|^n) = \ln |A|.
\]

As noted above, for any probability vector $(x_1, \ldots, x_n)$, $-\sum_{i=1}^n x_i \ln x_i \leq \ln n$, and equality is achieved if and only if all $x_i$ are equal to $\frac{1}{n}$. For any measure $\nu$ not equal to $\mu$ as in the lemma, there exists $n$ for which $\nu([w])$ is not uniform over all $w \in A^n$, and so 
\[
h(\nu) \leq \frac{-1}{n} \sum_{w \in A^n} \mu([w]) \ln \mu([w]) < \frac{1}{n} \ln (|A|^n) = \ln |A| = h(A^{\mathbb{Z}}).
\]

This implies that every $\nu \neq \mu$ is not a measure of maximal entropy, and so $\mu$ must be the only measure of maximal entropy on $A^{\mathbb{Z}}$.

\end{proof}

It is well-known that the specification property of Bowen (\cite{bowen}) implies uniqueness of the measure of maximal entropy. We consider two weakenings of Bowen's property, called \textbf{almost weak specification} and \textbf{almost specification}. Though these properties can be defined for arbitrary topological dynamical systems, we here restrict our attention to subshifts, giving definitions specific to that case which are slightly simpler.

The following property was originally defined in \cite{marcusmonatshefte}, Lemma 2.1, but was not there given a name.

\begin{definition}\label{AWSdef}

A subshift $X$ has \textbf{almost weak specification with gap function $f(n)$} if\\

\noindent
$\bullet$ $f(n)$ is positive and nondecreasing


\noindent
$\bullet$ $\frac{f(n)}{n} \rightarrow 0$

\noindent
$\bullet$ For any words $w^{(1)}$, $w^{(2)}$, $\ldots$, $w^{(k)} \in \mathcal{L}(X)$, and for any integers $n_1$, $\ldots$, $n_{k-1}$ where $n_i \geq f(|w^{(i)}|)$ for all $i$, there exist words $v^{(1)} \in \mathcal{L}_{n_1}(X)$, $\ldots$, $v^{(k-1)} \in \mathcal{L}_{n_{k-1}}(X)$ so that the word $w^{(1)} v^{(1)} w^{(2)} v^{(2)} \ldots w^{(k-1)} v^{(k-1)} w^{(k)} \in \mathcal{L}(X)$. 

\end{definition}

The assumption that $f(n)$ is nondecreasing is not explicitly required in the literature, but prevents some pathological cases which would make our proofs more complicated. 
We require $f(n)$ to be positive since, for any $n$, $f(n) = 0$ would imply that the higher power shift $X^n$ is just a full shift.

\begin{definition}
A subshift $X$ has \textbf{specification with gap $g$} if it has almost weak specificaton with the constant gap function $f(n) = g$.
\end{definition}


We note that this definition is slightly different from some definitions of specification in the literature, which often assume that the final word created is part of a periodic point. However, our version of specification implies denseness of periodic points, and so the definitions are equivalent for subshifts (see \cite{bertrand}).

A second weakening of specification was defined by Pfister and Sullivan in \cite{pfister-sullivan}, and was there called the $g$-almost product property. We follow the convention of \cite{thompson} and call this property \textbf{almost specification}.

\begin{definition}\label{ASdef}
A subshift $X$ has \textbf{almost specification with mistake function $g(n)$} if\\

\noindent
$\bullet$ $g(n)$ is positive and nondecreasing

\noindent
$\bullet$ $\frac{g(n)}{n} \rightarrow 0$

\noindent
$\bullet$ For any words $w^{(1)}$, $w^{(2)}$, $\ldots$, $w^{(k)} \in \mathcal{L}(X)$, there exist words $v^{(1)}$, $v^{(2)}$, $\ldots$, $v^{(k)} \in \mathcal{L}(X)$ so that $|w^{(i)}| = |v^{(i)}|$ for every $i$, $w^{(i)}$ and $v^{(i)}$ differ on at most $g(|w^{(i)}|)$ letters for every $i$, and the concatenation $v^{(1)} v^{(2)} \ldots v^{(k)}$ is in $\mathcal{L}(X)$.

\end{definition}

There are many subshifts known to satisfy almost specification; for instance, any $\beta$-shift has almost specification with gap function $g(n) = 1$ (see \cite{pfister-sullivan2}), and many of the so-called $S$-gap shifts also satisfy almost specification (with gap function dependent on $S$).

\begin{remark}
In \cite{yamamoto}, Yamamoto also studies various weakenings of specification and their implications. The property that he calls almost specification is our almost weak specification, and the property that he calls the almost product property is essentially our almost specification. 
\end{remark}

\section{Almost weak specification}\label{AWS}


\begin{proof}[Proof of Theorem~\ref{mainex}]

We define $X$ to have alphabet $A = \{-N,\ldots,-1,0,1,\ldots,N\}$, with $N \geq 450$, and list of forbidden words $\mathcal{F}$ consisting of:\\

\noindent
$\bullet$ All adjacent pairs $ij$ whose product is negative, i.e. consisting of one negative and one positive letter

\noindent
$\bullet$ All pairs $i0j$ with $ij$ whose product is negative

\noindent
$\bullet$ All words $v_1 v_2 \ldots v_j 0^{m} v_{j+1}$, where all $v_i \neq 0$ and $m < 2 + \log_3 j$.\\

$X$ then contains all points which look like $\ldots w^{(-1)} 0^{m_{-1}} w^{(0)} 0^{m_0} w^{(1)} 0^{m_1} \ldots$, where each $w^{(i)}$ consists of a ``run'' of nonzero letters of the same sign, and for every $i$, $m_i \geq 2 + \log_3 |w^{(i)}|$. All other points of $X$ are ``degenerate'' cases which have either an infinite or biinfinite string of $0$s or positives or negatives. We also note for future reference that, given any word $w \in A^*$ which does not contain a forbidden word as described in the above list, the point $\ldots 000w000 \ldots$ is clearly in $X$, and so $w \in \mathcal{L}(X)$. 


We note that this example is quite similar to an example of Haydn (\cite{haydn}), for which he also proved existence of two ergodic measures of maximal entropy with disjoint support. The main difference is that he required the length of a run of $0$ letters to be at least linear in the size of the nonzero runs to the left and right of it (thus precluding almost weak specification), and that he forced the signs of nonzero runs separated by a run of $0$ letters to be opposite. His proof involved a word-counting argument to prove that the topological entropy of his system was $\ln N$; our approach will be a bit more complicated and measure-theoretic.

First, we will show that $X$ has almost weak specification with gap function $2 + \lceil \log_3 n \rceil$. Consider any words $w^{(1)}, \ldots, w^{(k)}$ in $\mathcal{L}(X)$ and any positive $m_1, \ldots, m_{k-1}$ with $m_i \geq 2 + \lceil \log_3 |w^{(i)}| \rceil$ for each $i$. Define $w := w^{(1)} 0^{m_1} w^{(2)} 0^{m_2} \ldots 0^{m_{k-1}} w^{(k)}$; we claim that $w \in \mathcal{L}(X)$, which will demonstrate the desired almost weak specification. Firstly, since all $m_i$ are greater than or equal to $2$, introducing the runs of $0$ letters between the words $w^{(i)}$ could not have possibly introduced an adjacent pair of nonzero letters of opposite sign, or a pair of such letters separated by a single $0$. All that's left is to show that $w$ does not contain any word of the form $v_1 v_2 \ldots v_j 0^{m} v_{j+1}$ with all $v_i$ nonzero and $m < 2 + \log_3 j$. Suppose for a contradiction that $w$ does contain such a word, call it $u$. Clearly $u$ cannot be contained in any of the $w^{(i)}$, since they were assumed to be in $\mathcal{L}(X)$. Just as clearly, the central $0^m$ in $u$ must contain an entire $0^{m_i}$ from $w$, and the letters $v_1 \ldots v_j$ must all be from the suffix of some $w^{(i)}$. However, this means that $m \geq 2 + \log_3 |w^{(i)}|$, and since we assumed $m < 2 + \log_3 j$, it must be the case that $j > |w^{(i)}|$. But this is impossible; $w^{(i)}$ is preceded in $w$ by a $0$ for $i > 1$, and by nothing for $i = 1$. Therefore, $w$ contains no forbidden words, and so is in $\mathcal{L}(X)$, proving almost weak specification of $X$ with gap function $2 + \lceil \log_3 n \rceil$.

Of course this was not the desired rate of almost weak specification. However, we will now show that for any $C > 0$, there exists a higher-power shift $X^k$ with almost weak specification with gap function 
$\max(1, \lceil C \log_3 n \rceil)$. This will clearly imply the desired property since any positive $f$ with $\displaystyle \liminf_{n \rightarrow \infty} \frac{f(n)}{\ln n} > 0$ is bounded from below by 
$\max(1, \lceil C \log_3 n \rceil)$ for some $C > 0$.

Choose $C > 0$ and define $k = \lceil 8C^{-1} \rceil$; we note that $k \geq 8$. Consider $X^k$, the $k$th higher-power shift of $X$. The reader may check that the almost weak specification of $X$ with gap function $2 + \lceil \log_3 n \rceil$ implies almost weak specification of $X^k$ with gap function $\lceil (2 + \lceil \log_3(kn) \rceil)/k \rceil$. We claim that 
\begin{equation}\label{higherpowerbound}
\lceil (2 + \lceil \log_3(kn) \rceil)/k \rceil \leq \max(1, \lceil C \log_3 n \rceil)
\end{equation}
for all $n$, which would imply that $X^k$ has almost weak specification with gap function $\max(1, \lceil C \log_3 n \rceil)$ as desired.

To verify (\ref{higherpowerbound}), we first note that both sides are always positive, and so we do not need to prove anything in the case where the left-hand side is $1$. Suppose that the left-hand side is at least $2$. Then, $(2 + \lfloor \log_3(kn) \rfloor)/k > 1$, which implies that $\log_3(kn) > k-2$, and so $n > 3^{k-2}/k > k$, since $k \geq 8$. Finally, we notice that $\lceil \log_3(kn) \rceil \geq 2$, and so
\[
\frac{2 + \lceil \log_3(kn) \rceil}{k} \leq \frac{2}{k} \lceil \log_3(kn) \rceil \leq \frac{4}{k} \log_3(kn) \leq \frac{4}{k} \log_3 (n^2) = \frac{8}{k} \log_3 n < C \log_3 n.
\]

This implies (\ref{higherpowerbound}), and therefore that for any positive $f(n)$ with $\liminf_{n \rightarrow \infty} \frac{f(n)}{\ln n} > 0$, there exists $k$ for which $X^k$ has almost weak specification with gap function $f(n)$.\\

Now we will show that $X$ has exactly two ergodic measures of maximal entropy, whose supports are disjoint, and that this property holds for all higher-power shifts $X^k$ as well, which will complete the proof of Theorem~\ref{mainex}. Consider any ergodic measure of maximal entropy $\mu$ of $X$. Our goal is to show that $\mu([0]) = 0$. We will go about this by inductively proving the following claim: for every $k \geq 2$ and every word 
$w \in \mathcal{L}(X)$ with $\mu([w]) > 0$ which ends with a nonzero letter, 
\begin{equation}\label{indhyp}
\mu([w 0^k *] \ | \ [w]) < 4 (0.5N)^{-0.5(k-1)},
\end{equation}
where the event $[w 0^k *]$ is the union of $[w 0^k b]$ over all nonzero letters $b$. (In general, we will use $*$ in cylinder sets to denote the union over all nonzero letters at those locations, and in words to denote that the word in question could have any nonzero letters at those locations.)

We begin with the base case $k = 2$. Fix any word $w$ with $\mu([w]) > 0$ which ends with a nonzero letter, and denote $\mu([w]) = \beta$ and $\mu([w 00 *] \ | \ [w]) = \alpha$; then $\mu([w00*]) = \alpha \beta$. Then, for every $\epsilon > 0$, define the collection $C_{n,\epsilon}$ of words in $\mathcal{L}_n(X)$ with between $(\beta - \epsilon)n$ and $(\beta + \epsilon) n$ occurrences of $w$ and between $(\alpha\beta - \epsilon) n$ and $(\alpha\beta + \epsilon)n$ occurrences of $w 00 *$. By Lemma~\ref{genericcount}, there exists $M$ so that for any $n > M$, $|C_{n,\epsilon}| \geq e^{n(h(X) - \epsilon)}$. We will now use a replacement argument to show that if $\alpha$ is too large, then we could create new collections of words in $\mathcal{L}_n(X)$ with cardinality growing at exponential rate greater than $h(X)$, a contradiction.

For any $n > M$, our replacement process assigns, to every $v \in C_{n,\epsilon}$, a set $S(v) \subset \mathcal{L}_n(X)$ in the following way. First, enumerate the occurrences of $w 00 *$ (again, the $*$ can be any nonzero letter) in $v$ from left to right. We will only work with every other (i.e. first, third, fifth, etc.) of these occurrences, counting from the left. For each such occurrence $w00b$, we complete $b$ to the maximal run of nonzero letters that it is contained in, yielding a subword of $v$ of the form $w 00 u$. We then change the common sign of all letters in $u$ (leaving the absolute values unchanged) to match the sign of the final letter of $w$. Finally, we remove the $00$ and replace it by any of the $N^2$ pairs of letters in $A$ with the same sign as the last letter of $w$, and also change the final letter of $u$ to $0$, thus lengthening the run of $0$ letters after $u$ (which is not chosen for replacement since we only selected every other $w 00 *$) by one. Doing this independently for every other occurrence of $w 00 *$ yields our collection $S(v)$, which has cardinality at least $(N^2)^{0.5(\alpha\beta - \epsilon) n}$ and which we claim always consists of words in $\mathcal{L}_n(X)$.

We first must show that all words created in this way are in $\mathcal{L}(X)$. It should be reasonably clear that the only possible problem is creating the longer runs of nonzero letters, which could have length too large in comparison to the next run of $0$ letters. For a particular replaced occurrence of $00$, we consider the next run of $0$ letters to the right. If this next run of $0$ letters occurred at the rightmost edge of $v$, then the replacement cannot have created a forbidden word, since there is no nonzero letter to ``end'' this rightmost run of $0$ letters. We therefore assume that the run of $0$ letters after our replaced $00$ is ended by a nonzero symbol, and denote its length by $j$. Then, the length of the run of nonzero letters immediately to the left of the replaced $00$ has length exactly $1$ (since it appeared immediately to the left of $00$, and two consecutive nonzero letters would force $\lceil 2 + \log_3 2 \rceil = 3$ zero letters after them), and the run of nonzero letters immediately to the right of the replaced $00$ has length less than or equal to $3^{j-2}$ (since it appeared immediately to the left of a run of $0$ letters of length $j$). The run of nonzero letters created after the replacement therefore has length $\ell \leq 3^{j-2} + 1 + 1 \leq 3^{j-1}$, and is followed by a run of $0$ letters of length $j + 1$. Therefore, since $2 + \log_3 \ell \leq j+1$, the words in $S(v)$ are all in $\mathcal{L}_n(X)$. 

If the collections $S(v)$ were disjoint, this would yield an obvious lower bound on $|\mathcal{L}_n(X)|$, but they are definitely not. We will instead find an upper bound on, for any $u$, the number of $v$ for which $S(v)$ contains $u$. Fix any $u$ which is in any $S(v)$. If we wish to enumerate the $v$ satisfying $u \in S(v)$, much of the word $v$ is already determined by $u$; the only unknowns are\\

\noindent
$\bullet$ the locations of the $w 00 *$ which were chosen for replacement in $v$

\noindent
$\bullet$ the original signs of the nonzero runs following the chosen occurrences of $w 00 *$

\noindent
$\bullet$ the original absolute value of the final letter of each of the aforementioned runs\\

The total number of choices for these pieces of information is bounded from above by
\[
{(\beta + \epsilon) n \choose 0.5 (\alpha \beta + \epsilon) n} 2^{0.5 (\alpha \beta + \epsilon) n} N^{0.5 (\alpha \beta + \epsilon) n}.
\]

Therefore, the union of all $S(v)$ over $v \in C_{n,\epsilon}$ has cardinality at least
\[
e^{n(h(X) - \epsilon)} \frac{N^{2(0.5(\alpha\beta - \epsilon) n)}}{{(\beta + \epsilon) n \choose 0.5 (\alpha \beta + \epsilon) n} 2^{0.5 (\alpha \beta + \epsilon) n} N^{0.5 (\alpha \beta + \epsilon) n}}.
\]

Taking a logarithm and dividing by $n$ yields
\[
h(X) - \epsilon - 2 \epsilon \ln N + 0.5 (\alpha \beta + \epsilon) \ln(0.5N) - \frac{\ln {(\beta + \epsilon) n \choose 0.5 (\alpha \beta + \epsilon) n}}{n}.
\]

By Stirling's approximation, the limit of this expression as $n \rightarrow \infty$ is
\begin{multline}\label{entropycontrad}
h(X) - \epsilon - 2 \epsilon \ln N + 0.5 (\alpha \beta + \epsilon) \ln(0.5N) \\
+ (\beta + \epsilon) \Big[ 0.5\big((\alpha \beta + \epsilon)/(\beta + \epsilon)\big) \ln \big(0.5(\alpha \beta + \epsilon)/(\beta + \epsilon)\big) \\
+ \big(1 - (0.5(\alpha \beta + \epsilon)/(\beta + \epsilon))\big) \ln \big(1 - (0.5(\alpha \beta + \epsilon)/(\beta + \epsilon))\big)\Big].
\end{multline}

The limit of this expression as $\epsilon \rightarrow 0$ is
\[
h(X) + 0.5 \alpha \beta \ln(0.5N) + \beta(0.5\alpha \ln(0.5 \alpha) + (1 - 0.5\alpha)\ln(1 - 0.5\alpha),
\]

and so for a contradiction, it suffices to know that
\[
0.5 \alpha \beta \ln(0.5N) > \beta (-0.5 \alpha \ln (0.5 \alpha) - (1 - 0.5 \alpha) \ln (1 - 0.5 \alpha))
\]
or, equivalently,
\begin{equation}\label{entropycontrad2}
0.5 \alpha \ln(0.5N) > -0.5 \alpha \ln (0.5 \alpha) - (1 - 0.5 \alpha) \ln (1 - 0.5 \alpha),
\end{equation}
since then we could take $\epsilon$ small enough to make (\ref{entropycontrad}) greater than $h(X) + \delta$ for some $\delta > 0$, then take $n$ large enough that our replacement procedure yields more than $e^{n(h(X) + \delta)}$ words in $\mathcal{L}_n(X)$ for all large $n$, a contradiction. 

We make the change of variable $\alpha' = 0.5 \alpha$; clearly $\alpha' \leq 0.5$. Then, it's easily verified that
$-(1 - \alpha') \ln(1 - \alpha') < -\alpha' \ln \alpha'$, and so we will have a contradiction as long as
\begin{multline*}
\alpha' \ln(0.5N) \geq -2\alpha' \ln \alpha' \Leftrightarrow \ln(0.5N) \geq -2 \ln \alpha' \Leftrightarrow \\
\alpha' \geq (0.5N)^{-0.5} \Leftrightarrow \alpha \geq 2 (0.5N)^{-0.5}.
\end{multline*}

Therefore, we have shown that $\alpha \geq 2 (0.5N)^{-0.5}$ yields a contradiction, and so we know that $\alpha = \mu([w 00 *] | [w]) < 2 (0.5N)^{-0.5}$, verifying (\ref{indhyp}) for $k = 2$.\\

We now choose any $k > 2$, and assume (\ref{indhyp}) for all $w$ with $\mu([w]) > 0$ which end with a nonzero letter and all smaller values than $k$. For any $w$, we wish to perform a similar replacement procedure to the above for occurrences of $w 0^k *$. However, there is a serious problem. Suppose that $k$ is large, that $w$ ends with a string of nonzero letters of length $3^{k-2}$, and that the next run of $0$ letters is very short, maybe of length two, meaning that we have a word of the form $u 0^k * 00 *$, $|u| = 3^{k-2}$. In mimicking the previous replacement argument, we wish to replace $0^k$ with nonzero letters, and so we will then have a run of nonzero letters of length approximately $3^{k-2}$ before $00$, which is clearly illegal. To fix this, we would need to add at least $k-2$ $0$ letters to increase the length of the next run to at least $k$. This means that informally, we've made a ``gain'' from replacing $k$ copies of $0$ by one of $N$ positives or negatives, but we've ``lost'' information from $k-2$ positives or negatives replaced by $0$ letters. This means that we will not get a large enough net gain by our replacements to get a useful bound of the sort in (\ref{indhyp}). 

However, it turns out that if the following run of $0$ letters has length at least $k$, then we can get away with adding only a single $0$ letter to the next run. All of this means that another replacement argument can be used to show that the conditional probability, given $w$, of $w 0^k *$ followed by a run of $0$ letters with length at least $k$, is very small (on the order of (\ref{indhyp}).) The final trick we use is to show, using the inductive hypothesis, that given $w 0^k *$, the conditional probability that the next run of $0$ letters has length at least $k$ is greater than $\frac{1}{2}$. This means that the conditional probability, given $w$, of $w 0^k *$, is at most twice the conditional probability, given $w$, of $w 0^k *$ where the next run of $0$ letters has length at least $k$, which we have already bounded from above. We now describe these steps formally.

Choose any $w \in \mathcal{L}(X)$ with $\mu([w]) > 0$, and denote $\beta = \mu([w])$ and $\alpha = \mu([w 0^k - 0^{\geq k} *] \ | \ [w])$, where $[w 0^k - 0^{\geq k} *]$ represents the event that $w 0^k$ is followed by some run of positives or negatives, then by a run of $0$ letters of length at least $k$, followed by a nonzero letter. Again, $\mu([w 0^k - 0^{\geq k} *]) = \alpha \beta$. For any $\epsilon > 0$, we define the collection $C_{n,\epsilon}$, which now will consist of words in $\mathcal{L}_n(X)$ with between $(\beta - \epsilon) n$ and $(\beta + \epsilon) n$ occurrences of $w$ and between $(\alpha\beta - \epsilon) n$ and $(\alpha \beta + \epsilon) n$ occurrences of $w 0^k - 0^{\geq k} *$. Again by Lemma~\ref{genericcount}, there exists $M$ so that for any 
$n > M$, $|C_{n,\epsilon}| \geq e^{n(h(X) - \epsilon}$. We now proceed as before, for any $n > M$ performing a replacement procedure on each $v \in C_{n,\epsilon}$ by removing $0^k$ in every other occurrence of $w 0^k - 0^{\geq k} *$, adjusting the sign of the nonzero run $u$ after $0^k$ to match that of the nonzero run before $0^k$, and changing only the final letter of $u$ to $0$, thus lengthening the run of $0$ letters after $u$ by one. This again yields a set $S(v)$ of words, which we will now show are in $\mathcal{L}_n(X)$. 

It should be reasonably clear that, as before, the only possible problem is creating the longer runs of nonzero letters, which could have length too great for the next run of $0$ letters. Again, for a particular replaced occurrence of $0^k$, we denote by $j$ the length of the next run of $0$ letters to the right; note that 
$j \geq k$. Then, the length of the run of nonzero letters immediately to the left of the replaced $0^k$ has length less than or equal to $3^{k-2}$ (since it appears immediately to the left of $0^k$), and the run of nonzero letters immediately to the right of the replaced $0^k$ has length less than or equal to $3^{j-2}$ (since it appears immediately to the left of $0^j$). The run of nonzero letters created after the replacement therefore has length $\ell \leq 3^{k-2} + 3^{j-2} + k - 1 \leq 2 \cdot 3^{j-2} + j - 1 \leq 3^{j-1}$ 
(since $j \geq k \geq 3$), and is followed by a run of $0$ letters of length $j + 1$. Therefore, since $2 + \log_3 \ell \leq j + 1$, the words in $S(v)$ are all in $\mathcal{L}_n(X)$. 

It is easily verified using the same techniques as were done for $k = 2$ that each $S(v)$ has cardinality at least $(N^k)^{0.5(\alpha\beta - \epsilon)n}$, and that for any $u$ in some $S(v)$, the number of $v$ for which $S(v)$ contains $u$ is still at most 
\[
{(\beta + \epsilon) n \choose 0.5 (\alpha \beta + \epsilon) n} 2^{0.5 (\alpha \beta + \epsilon) n} N^{0.5 (\alpha \beta + \epsilon) n)}.
\]
Exactly as before (i.e. take logarithms, divide by $n$, and let $n \rightarrow \infty$ and $\epsilon \rightarrow 0$), this will yield a contradiction if
\[
0.5 \alpha (k-1) \ln(0.5N) > (-0.5 \alpha \ln (0.5 \alpha) - (1 - 0.5 \alpha) \ln (1 - 0.5 \alpha)).
\]
We again make the change of variable $\alpha' = 0.5 \alpha$ and note that $\alpha' \leq 0.5$. This means that we can derive a contradiction as long as 
\begin{multline*}
\alpha' (k-1) \ln(0.5N) \geq -2\alpha' \ln \alpha' \Leftrightarrow (k-1) \ln(0.5N) \geq -2 \ln \alpha' \Leftrightarrow\\
\alpha' \geq (0.5N)^{-0.5(k-1)} \Leftrightarrow \alpha \geq 2 (0.5N)^{-0.5(k-1)}.
\end{multline*}
We have therefore showed that $\alpha = \mu([w 0^k - 0^{\geq k} *] \ | \ [w]) < 2 (0.5N)^{-0.5(k-1)}$. 

We now wish to bound from above the conditional probability $\mu([w 0^k - 0^{<k} *] \ | \ [w - 0^k])$, where $[w 0^k - 0^{<k} *]$ represents the event that 
$w$ is followed by some nonzero letter, then by a run of $0$ letters of length $k$, then by some run of positives or negatives, then by a run of $0$ letters of length less than $k$ ended by a nonzero letter. We break this event up as a disjoint union:
\[
[w 0^k - 0^{<k} *] = \bigsqcup_{i=2}^{k-1} \bigsqcup_{j=1}^{3^{i-2}} \bigsqcup_{\substack{|v| = j \\ \forall m, v_m \neq 0}} [w 0^k v 0^i *],
\] 
where the third union is over $v$ of length $j$ which consist entirely of positives or negatives.
Then,
\begin{multline*}
\mu([w 0^k - 0^{<k} *] \ | \ [w 0^k *]) = \sum_{i=2}^{k-1} \sum_{j=1}^{3^{i-2}} \sum_{\substack{|v| = j \\ \forall m, v_m \neq 0}} \mu([w 0^k v 0^i *] \ | \ [w 0^k *])\\
= \sum_{i=2}^{k-1} \sum_{j=1}^{3^{i-2}} \sum_{\substack{|v| = j \\ \forall m, v_m \neq 0}} \mu([w 0^k v] \ | \ [w 0^k *]) \mu([w 0^k v 0^i *] \ | \ [w 0^k v])\\
\leq \sum_{i=2}^{k-1} \sum_{j=1}^{3^{i-2}} \sum_{\substack{|v| = j \\ \forall m, v_m \neq 0}} 4 (0.5N)^{-0.5(i-1)} \mu([w 0^k v] \ | \ [w 0^k *])
\leq 4 \sum_{i=2}^{k-1} \sum_{j=1}^{3^{i-2}} (0.5N)^{-0.5(i-1)}\\
\leq 4 \sum_{i=2}^{\infty} 3^{i-2} (0.5N)^{-0.5(i-1)} = \frac{(4/3)C}{1-C},
\end{multline*}
where $C = \frac{3}{\sqrt{0.5N}}$. (The first inequality came from applying the inductive hypothesis (\ref{indhyp}) to $w 0^k v$.) Since $N \geq 450$, $C \leq \frac{1}{5}$. The reader can check that then $\frac{C}{1-C} \leq \frac{1}{4}$, and so 
$\mu([w 0^k - 0^{<k} *] \ | \ [w 0^k *]) \leq \frac{1}{3}$, and so $\mu([w 0^k - 0^{\geq k} *] \ | \ [w 0^k *]) \geq \frac{2}{3} > \frac{1}{2}$.

This means that
\begin{multline*}
2 (0.5N)^{-0.5(k-1)} \geq \mu([w 0^k - 0^{\geq k} *] \ | \ [w]) \\ = \mu([w 0^k - 0^{\geq k} *] \ | \ [w 0^k *]) \mu([w 0^k *] \ | \ [w]) > 0.5 \mu([w 0^k *] \ | \ [w]),
\end{multline*}
implying $\mu([w 0^k *] \ | \ [w]) < 4 (0.5N)^{-0.5(k-1)}$, verifying (\ref{indhyp}) for $k$ and completing the proof by induction that (\ref{indhyp}) holds for all 
$k$ and all $w$ with $\mu([w]) > 0$ which end with a nonzero letter.

One can take a weighted average over all $w$ to see that (\ref{indhyp}) implies that for all $k$,
\[
\mu([* 0^k *]) < 4 (0.5N)^{-0.5(k-1)},
\]
where $[* 0^k *]$ is the union of $[a 0^k b]$ over all nonzero $a,b$.

Finally, let's suppose for a contradiction that $\mu([0]) > 0$. Then, by Birkhoff's ergodic theorem (Theorem~\ref{birkhoff}) applied to $f = \chi_{[0]}$, for $\mu$-a.e. $x \in X$,
\[
\lim_{n \rightarrow \infty} \frac{1}{2n+1} |\{i \in [-n,n] \ : \ x_i = 0\}| = \mu([0]) > 0.
\]
In particular, $\mu$-a.e. $x \in X$ contains infinitely many $0$s, and so we can partition all of $X$ except a measure zero subset $N$ by the location of the origin relative to the next consecutive run of $0$s:
\[
X = N \sqcup ([0] \cap N^c) \sqcup ([*] \cap N^c) = N \sqcup \left( \bigsqcup_{i=2}^{\infty} \bigsqcup_{j=1}^i \sigma^j [* 0^i *] \right) \sqcup 
\left( \bigsqcup_{i=2}^{\infty} \bigsqcup_{j=1}^{3^{i-2}} \sigma^j \left( \bigsqcup_{\substack{|v| = j \\ \forall m, v_m \neq 0}} [v 0^i] \right) \right),
\]
where again each $*$ represents any nonzero letter.

This implies that 
\begin{multline*}
\mu(X) = 1 = \left( \sum_{i=2}^{\infty} \sum_{j=1}^i \mu([* 0^i *]) \right) + 
\left( \sum_{i=2}^{\infty} \sum_{j=1}^{3^{i-2}} \sum_{\substack{|v| = j \\ \forall m, v_m \neq 0}} \mu([v 0^i *]) \right)\\
\leq \sum_{i=2}^{\infty} 4i (0.5N)^{-0.5(i-1)} + \sum_{i=2}^{\infty} \sum_{j=1}^{3^{i-2}} \sum_{\substack{|v| = j \\ \forall m, v_m \neq 0}} 4(0.5N)^{-0.5(i-1)} \mu([v])\\
\leq \sum_{i=2}^{\infty} 4(i + 3^{i-2}) (0.5N)^{-0.5(i-1)} < 12 \sum_{i=2}^{\infty} 3^{i-2} (0.5N)^{-0.5(i-1)} = \frac{4C}{1-C} \leq 1,
\end{multline*}
a contradiction. (Again, the first inequality came from applying (\ref{indhyp}) to $v$, and $C = \frac{3}{\sqrt{0.5N}}$, which we know to be less than or equal to $\frac{1}{5}$.) Therefore, we have finally shown that $\mu([0]) = 0$ for all ergodic measures of maximal entropy $\mu$ on $X$. The ergodic decomposition implies that in fact this is true for nonergodic measures of maximal entropy on $X$ as well.

However, if $\mu([0]) = 0$, then $\mu$ has support contained within the union of the full shift on $\{1, \ldots N\}$ and the full shift on $\{-1, \ldots, -N\}$. Since both of these full shifts are shift-invariant, if $\mu$ is in addition ergodic, its support is contained in one or the other. Since each of these shifts has a unique ergodic measure of maximal entropy by Lemma~\ref{fsunique} (namely the uniform Bernoulli measure), and since those measures have the same entropy $\ln N$, we have shown that $h(X) = \ln N$ and that the only ergodic measures of maximal entropy are the uniform Bernoulli measures on the positives and negatives respectively, which have disjoint support.

We now prove the same about any higher power shift $X^k$. Consider any ergodic measure of maximal entropy $\nu$ on $X^k$. Then $\nu$ clearly induces a measure $\nu^*$ on $X$ by defining $\nu^*([w^{(1)} \ldots w^{(n)}]) = 
\nu([w^{(1)} \ldots w^{(n)}])$ for all choices of $w^{(1)}, \ldots, w^{(n)} \in \mathcal{L}_k(X)$; the $w^{(i)}$ are interpreted as concatenated $k$-letter words on the left-hand side and as letters in the alphabet of $X^k$ on the right-hand side. It's quite possible that $\nu^*$ is not even $\sigma$-invariant, but it is invariant under $\sigma^k$ since $\nu$ was $\sigma$-invariant on $X^k$. Therefore, $\mu := \frac{1}{k} \sum_{i=0}^{k-1} \sigma^i \nu^*$ is a $\sigma$-invariant measure on $X$, and it is a measure of maximal entropy on $X$ since $\nu$ was a measure of maximal entropy on $X^k$ and the entropy map $\mu \mapsto h(\mu)$ is affine. Therefore, $\mu([0]) = 0$, which clearly implies that $\nu([u]) = 0$ for every $u \in \mathcal{L}_k(X)$ containing a $0$ letter. Now, since $\nu$ is ergodic as a measure on $X^k$, this implies that $\nu$ is supported either entirely on the full shift on $\{1, \ldots N\}^k$ or the full shift on $\{-1, \ldots, -N\}^k$, and as before this means that there are exactly two choices for $\nu$.

\end{proof}


\begin{proof}[Proof of Theorem~\ref{mainthm}]
Suppose for a contradiction that $X$ is a subshift with almost weak specification with a gap function $f(n)$ satisfying $\liminf_{n \rightarrow \infty} \frac{f(n)}{\ln n} = 0$ and which possesses two measures of maximal entropy $\mu, \nu$ with disjoint supports. 

We first use $\mu$ and $\nu$ to give a lower bound on $|\mathcal{L}_n(X)|$. For every $n$, define $\mathcal{M}_n(X) = \{w \in \mathcal{L}_n(X) \ : \ \mu([w]) > 0\}$ and $\mathcal{N}_n(X) = \{w \in \mathcal{L}_n(X) \ : \ \nu([w]) > 0\}$. (For $n = 0$, we define both $\mathcal{M}_n(X)$ and $\mathcal{N}_n(X)$ to be singletons consisting of the empty word $\varnothing$.) It should be clear that for any $k < n$, any $k$-letter subword of a word in $\mathcal{M}_n(X)$ must be in $\mathcal{M}_k(X)$, and that a similar statement holds for subwords of words in $\mathcal{N}_n(X)$. Lemma~\ref{supportcount} implies that $|\mathcal{M}_n(X)| \geq e^{n h(\mu)} = e^{n h(X)}$ and $|\mathcal{N}_n(X)| \geq e^{n h(\mu)} = e^{n h(X)}$ for all $n$. Also, since $\mu$ and $\nu$ have disjoint supports, there exists $N$ so that $\mathcal{M}_n(X) \cap \mathcal{N}_n(X) = \varnothing$ for all $n \geq N$. 


We choose any $n$ for which $n > f(n)$ (possible since $\frac{f(n)}{n} \rightarrow 0$) and will use almost weak specification to bound $|\mathcal{L}_n(X)|$ from below as follows: for any $i \in [0,n - f(n)]$ which is a multiple of 
$f(n) + N$, choose words $w \in \mathcal{M}_i(X)$ and $v \in \mathcal{N}_{n - f(n) - i}(X)$. Then, since $i, n - i - f(n) \leq n$ and since $f(n)$ is nondecreasing, almost weak specification of $X$ implies that there exists $u$ with length $f(n)$ so that $wuv \in \mathcal{L}_{n}(X)$. We claim that the map from $(i,w,v)$ to $wuv$ (choose $u$ to be the lexicographically minimal option to make this map a function) is one-to-one. To see this, suppose for a contradiction that for choices $(i,w,v) \neq (i', w', v')$, $wuv = w'u'v'$. If $i = i'$, then either $w \neq w'$ or $v \neq v'$, and we have an obvious contradiction. But, if $i \neq i'$, then $|i - i'| \geq f(n) + N$, implying that either $w$ and $v'$ share an $N$-letter subword or $w'$ and $v$ share an $N$-letter subword, both contradictions since that word would be in $\mathcal{M}_{N}(X) \cap \mathcal{N}_{N}(X)$. 
Therefore, the map is one-to-one, and so generates at least 
\begin{multline*}
\sum_{i \in [0, n - f(n)], (f(n) + N) | i} |\mathcal{M}_i(X)| |\mathcal{N}_{n - f(n) - i}(X)| \geq \\
\sum_{i \in [0, n - f(n)], (f(n) + N) | i} e^{ih(X)} e^{(n - f(n) - i) h(X)} \geq \frac{n - f(n)}{f(n) + N} e^{(n - f(n)) h(X)}
\end{multline*}
words in $\mathcal{L}_{n}(X)$. Therefore, 
\begin{equation}\label{toomanywords}
|\mathcal{L}_{n}(X)| \geq \frac{n - f(n)}{f(n) + N} e^{(n - f(n)) h(X)}. 
\end{equation}

We also note that for any $t$ and any $w^{(1)}$, $\ldots$, $w^{(t)} \in \mathcal{L}_n(X)$, we can use almost weak specification to create a word 
$w^{(1)} v^{(1)} w^{(2)} v^{(2)} \ldots w^{(t)} v^{(t)}$ in $\mathcal{L}_{t(n+f(n))}(X)$, where all $v^{(t)}$ are of length $f(n)$. This clearly implies that 
\[
|\mathcal{L}_{t(n+f(n))}(X)| \geq |\mathcal{L}_n(X)|^t,
\]
and we can take logarithms, divide by $t$, and let $t$ approach infinity to see that
\begin{equation}\label{entbound}
(n + f(n)) h(X) \geq \ln |\mathcal{L}_n(X)|.
\end{equation}

Combining (\ref{toomanywords}) and (\ref{entbound}) implies that for large enough $n$,
\[
(n + f(n)) h(X) \geq \ln |\mathcal{L}_{n}(X)| \geq \ln(n - f(n)) - \ln(f(n) + N) + (n - f(n)) h(X).
\]

We rewrite as
\[
2f(n) h(X) + \ln(f(n) + N) \geq \ln(n - f(n)).
\]

However, if we choose a sequence $n_k$ along which $\frac{f(n_k)}{\ln n_k} \rightarrow 0$ and let $k \rightarrow \infty$, then all terms on the left-hand side are 
$o(\ln(n_k))$, and the right side gets arbitrarily close to $\ln(n_k)$. Therefore, our original assumption was false, completing the proof.

\end{proof}

\section{Almost specification}\label{AS}

We must begin with some lemmas related to coding theory, namely constructions of small sets which are $n$-spanning with respect to the Hamming distance.

\begin{definition}
For any alphabet $A$ and $n \in \mathbb{N}$, the \textbf{Hamming distance} $d$ on $A^n$ is given by $d(v,w) := |\{i \ : \ v_i \neq w_i\}|$, the number of locations at which $v$ and $w$ differ.
\end{definition}

\begin{lemma}\label{codinglemma}
For every alphabet $A$ and positive integer $n$, there exists a set $T_{A,n} \subset A^n$ such that $|T_{A,n}| \leq \frac{1}{2^{\lfloor \log_2 n \rfloor}} |A|^n$ and $T_{A,n}$ is $1$-spanning with respect to the Hamming distance $d$, i.e. for any $w \in A^n$, there exists $t \in T_{A,n}$ s.t. $d(t,w) \leq 1$.
\end{lemma}

\begin{proof}

Choose any $A$ and $n$, and assume without loss of generality that $A = \{0, \ldots, |A|-1\}$. Define $m = \lfloor \log_2 n \rfloor$, so that $2^m \leq n < 2^{m+1}$. Then, for any $v = v_0 \ldots v_{m-1} \in \{0,1\}^m$, define $T_{A,n,v}$ to be the set of all $w = w_0 \ldots w_{n-1} \in A^n$ such that for every $j \in [0,m)$, the sum of $w_i$ over all $i \in [0,2^m-1]$ whose binary expansion has a $0$ in the $2^j$ place is equal to $v_j \pmod 2$. For example, take $A = \{0,1,2\}$ and $n = 10$ (so $m = 3$), and $v = 010$. Then, $T_{A,n,v}$ is the set of all $w \in A^n$ for which $w_0 + w_2 + w_4 + w_6 = 0 \pmod 2$, $w_0 + w_1 + w_4 + w_5 = 1 \pmod 2$, and $w_0 + w_1 + w_2 + w_3 = 0 \pmod 2$, and so $0121200111 \in T_{A,n,v}$ and $0211221100 \notin T_{A,n,v}$. 

We claim that any set $T_{A,n,v}$ is $1$-spanning. To see this, consider any $w \in A^n$. Then, for some values of $j \in [0,m)$, the sum of $w_i$ over all $i$ whose binary expansion has a $0$ in the $2^j$ place is already equal to $v_j \pmod 2$, and for some it is not. Define $J \subseteq [0,m)$ to be the set of $j$ for which the aforementioned sum is equal to $v_j \pmod 2$. Then, choose $i \in [0,2^m)$ so that the binary expansion of $i$ has
$0$s precisely in $2^j$-indexed places for $j \notin J$, i.e. $i = \sum_{j \in J} 2^j$. Note that $0 \leq i \leq 2^m - 1$. We can then define $w'$ to be any word obtained by changing $w_i$ to any letter of $A$ with the opposite parity; then we claim that $w' \in T_{A,n,v}$. This is because the sum of $w'_i$ over all $i$ whose binary expansion has a $0$ in the $2^j$ place is equal to the corresponding sum of $w_i$ if and only if $j \in J$, and so this sum will now always equal $v_j \pmod 2$. 

For instance, continuing the example above: $w = 0211221100 \notin T_{A,n,v}$. In this case, $w_0 + w_2 + w_4 + w_6 = 0 \pmod 2$, $w_0 + w_1 + w_4 + w_5 \neq 1 \pmod 2$, and $w_0 + w_1 + w_2 + w_3 = 0 \pmod 2$, so $J = \{0,2\}$. Then, we would define $i = 2^0 + 2^2 = 5$, and define $w'$ by changing $w_5 = 2$ to a letter of $A$ with opposite parity, so $w' = 0211211100$. Then, $w' \in T_{A,n,v}$. 

We finish by noting that $T_{A,n,v}$ is a partition of $A^n$, and so since there are $2^m$ choices for $v$, there exists $T_{A,n,v}$ with cardinality at most $\frac{1}{2^m} |A|^n$; define $T_{A,n}$ to be that set.

\end{proof}

\begin{remark}
The sets $T_{A,n}$ are essentially truncated Hamming codes on general alphabets. To say a bit more, the case where $A = \{0,1\}$ and $n$ is a power of $2$ (say $n = 2^m$) is special; it is one of the few cases where a ``perfect'' code is known to exist, i.e. a set $C$ which is $1$-spanning and for which every $w \in A^n$ has a unique $t \in C$ for which $d(w,t) = 1$. This is the Hamming code, and it coincides with our construction exactly for such $n$ and $A$ with $v = 0 \ldots 0$. 

Since we need such sets for all lengths and alphabets, we simply chose, for any $n$, the largest power of $2$ less than or equal to $n$ (i.e. $2^m$), and used a Hamming code on the first $2^m$ digits. We also used the same parity check idea even for larger alphabets where it is not nearly as efficient, since it still suffices for our purposes.
\end{remark}

\begin{lemma}\label{codinglemma2}

For every alphabet $A$ and positive integer $n$, there exists a set $U_{A,n} \subset A^n$ such that $|U_{A,n}| \leq \frac{16}{n^2} |A|^n$ and $U_{A,n}$ is $2$-spanning with respect to the Hamming distance $d$, i.e. for all $w \in A^n$, there exists $u \in U_{A,n}$ s.t. $d(u,w) \leq 2$.

\end{lemma}

\begin{proof}

For any $n$, we simply define $T_{A,\lfloor 0.5n \rfloor}$ and $T_{A, \lceil 0.5 n \rceil}$ as above, and define $U_{A,n} = \{uv \ : \ u \in T_{A,\lfloor 0.5n \rfloor}, v \in T_{A,\lceil 0.5n \rceil}\}$. It should be clear that $U_{A,n}$ is $2$-spanning; for any $w \in A^n$, at most one change is required to change its prefix of length $\lfloor 0.5n \rfloor$ to a word in $T_{A,\lfloor 0.5n \rfloor}$, and at most one change is required to change its suffix of length $\lceil 0.5n \rceil$ to a word in $T_{A,\lceil 0.5n \rceil}$. It's not hard to check that $2^{\lfloor \log_2 (\lfloor 0.5 n \rfloor) \rfloor} \geq 0.25n$. Then, by Lemma~\ref{codinglemma},
\[
|U_{A,n}| = |T_{A,\lfloor 0.5n \rfloor}| |T_{A,\lceil 0.5n \rceil}| \leq \frac{1}{0.25n} |A|^{\lfloor 0.5n \rfloor} \cdot \frac{1}{0.25n} |A|^{\lceil 0.5n \rceil} = \frac{16}{n^2} |A|^n.
\]

\end{proof}

\begin{lemma}\label{codinglemma2.5}
For every alphabet $A$ and positive integer $n > 1$, there exists a set $V_{A,n} \subset A^n$ such that $|V_{A,n}| = \frac{1}{|A|^2} |A|^n$ and $V_{A,n}$ is $2$-spanning with respect to the Hamming distance $d$.
\end{lemma}

\begin{proof}
Simply define $V_{A,n} = \{w = w_1 w_2 \ldots w_n \in A^n \ : \ w_1 = w_2 = 1\}$. The reader may check that $V_{A,n}$ has the desired properties.
\end{proof}

\begin{proof}[Proof of Theorem~\ref{mainex2}]

We define $X$ to have alphabet $A = \{-N,\ldots,-1,1,\ldots,N\}$, where $N$ will be defined later. For a parameter $\ell$, also to be determined later, we define 
\[
P_n = 
\begin{cases}
\{1\} & n = 1\\
V_{\{1,\ldots,N\},n} & 1 < n \leq \ell\\
U_{\{1,\ldots,N\},n} & n > \ell
\end{cases}
\]
and
\[
N_n = 
\begin{cases}
\{-1\} & n = 1\\
V_{\{-N,\ldots,-1\},n} & 1 < n \leq \ell\\
U_{\{-N,\ldots,-1\},n} & n > \ell
\end{cases}
\]
where $U_{A,n}$ and $V_{A,n}$ are defined as in Lemmas~\ref{codinglemma2} and~\ref{codinglemma2.5}. Clearly then $|N_n| = |P_n|$ for all $n$; we denote their common value by $M_n$. We also note that all $N_n$ and $P_n$ are $2$-spanning; for the case $n = 1$ this is trivial, and for $n > 1$ this comes from Lemmas~\ref{codinglemma2} and~\ref{codinglemma2.5}.

Then, we define $X$ via the list of forbidden words $\mathcal{F}$ consisting of \\

\noindent
$\bullet$ All words $n P n'$ where $n,n' < 0$, $P$ consists of positive letters, and $P \notin \bigcup P_n$ and

\noindent
$\bullet$ All words $p N p'$ where $p,p' < 0$, $N$ consists of negative letters, and $N \notin \bigcup N_n$.\\

Points of $X$ then consist of biinfinite concatenations of words of constant sign, each of which must be in either some $P_n$ or $N_n$, depending on its sign and length. There are also transient points of $X$ which have one or more infinite words of constant sign, on which there are no restrictions.

We first show that $X$ has almost specification with gap function $g(n) = 4$. For this purpose, consider any words $w^{(1)}$, $\ldots$, $w^{(k)}$. Then, the concatenation $w = w^{(1)} w^{(2)} \ldots w^{(k)}$ might not be in $\mathcal{L}(X)$. However, we can turn this into a word in $\mathcal{L}(X)$ by making changes to each maximal word of constant sign within $w$ which place them in either some $P_n$ or $N_n$. By the $2$-spanning property of all $P_n$ and $N_n$, we can change at most $2$ letters in each maximal word of constant sign within $w$ and create a new word $w' \in \mathcal{L}(X)$. No maximal word of constant sign which is not a prefix or suffix of $w^{(i)}$ would have required a change, since $w^{(i)} \in \mathcal{L}(X)$, implying that any such word would have been in some $P_n$ or $N_n$ anyway. 

Therefore, when $w$ was changed to $w'$, no more than $4$ changes would have been made in any $w^{(i)}$, those being only in the words of constant sign at the beginning and end at $w^{(i)}$. This completes the proof of almost specification with $g(n) = 4$.\\

We will now show that $h(X) = \ln N$, which will imply that $X$ has two measures of maximal entropy with disjoint supports, namely the uniform Bernoulli measures on the positive and negative letters of $A$ respectively. For this, we will just bound $\mathcal{L}_n(X)$ from above for all $n$. Every $w \in \mathcal{L}_n(X)$ consists of a concatenation of words of constant sign. Therefore, we can parametrize elements of $\mathcal{L}_n(X)$ by the number $k \geq 1$ of such concatenated words and their lengths $n_1$, $n_2$, $\ldots$, $n_k$, which clearly must sum to $n$. We then see that
\begin{equation}\label{firststep}
|\mathcal{L}_n(X)| = 2N^n + 2(n-1) N^n + \sum_{k=3}^n \sum_{\substack{n_1, \ldots, n_k\\ \sum n_i = n}} 2 N^{n_1} \left(\prod_{j=2}^{k-1} M_j\right) N^{n_k}.
\end{equation}

Here, the first two terms correspond to the cases $k = 1,2$. In each term, the factor of $2$ comes from choosing the sign of the first word of constant sign, after which all signs are forced. The $N^{n_1}$ and $N^{n_k}$ in the third term are from the prefix and suffix of $w$ of constant sign, on which there are no restrictions, and the $M_j$ represent the number of choices for the other subwords of constant sign. We now bound the third term of (\ref{firststep}) from above.
\begin{multline}\label{secondstep}
\sum_{k=3}^{n} \sum_{\substack{n_1, \ldots, n_k\\ \sum n_i = n}} 2 N^{n_1} \left(\prod_{j=2}^{k-1} M_j\right) N^{n_k} = 
2N^n \sum_{k=3}^{n} \sum_{\substack{n_1, \ldots, n_k\\ \sum n_i = n}} \prod_{j=2}^{k-1} (M_j N^{-j}) = \\
2N^n \sum_{k=3}^{n} \sum_{\substack{n_2, \ldots, n_{k-1}\\ \sum n_i < n}} \left(n - \sum_{j=2}^{k-1} n_j\right) \prod_{j=2}^{k-1} (M_j N^{-j}) \leq
2nN^n \sum_{k=3}^{n} \sum_{\substack{n_2, \ldots, n_{k-1}\\ \sum n_i < n}} \prod_{j=2}^{k-1} (M_j N^{-j}) \leq \\
2nN^n \sum_{k=3}^{n} \left( \sum_{t=1}^{\infty} M_t N^{-t} \right)^{k-2}.
\end{multline}

We now bound $\sum_{t=1}^{\infty} M_t N^{-t}$ by using the bounds of Lemmas~\ref{codinglemma2} and~\ref{codinglemma2.5}:

\begin{multline*}
\sum_{t=1}^{\infty} M_t N^{-t} = \frac{1}{N} + \sum_{t=2}^{\ell} |V_{\{1,\ldots,N\},t}| N^{-t} + \sum_{t=\ell + 1}^{\infty} |U_{\{1,\ldots,N\},t}| N^{-t} \leq \\
N^{-1} + \sum_{t=2}^{\ell} N^{-2} + \sum_{t=\ell + 1}^{\infty} 16t^{-2} \leq N^{-1} + (\ell - 1) N^{-2} + 16\ell^{-1}.
\end{multline*}

Choose $N$ and $\ell$ to be any pair for which this expression is less than $1$, for instance $N = 10$ and $\ell = 32$ (then 
$N^{-1} + (\ell - 1) N^{-2} + 16\ell^{-1} = 0.1 + 0.31 + 0.5 = 0.91 < 1$), and denote $\sum_{t=1}^{\infty} M_t N^{-t}$ by $\alpha < 1$. Then, 
(\ref{firststep}) and (\ref{secondstep}) imply that 
\[
|\mathcal{L}_n(X)| \leq 2N^n + 2(n-1)N^n + 2nN^n \sum_{k=3}^{\infty} \alpha^{k-2} = \frac{2n}{1-\alpha} N^n.
\]
This clearly implies that $h(X) \leq \ln N$, and so that $h(X) = \ln N$ since $X$ contains the full shifts on the $N$ positive and $N$ negative letters, each of which has topological entropy $\ln N$. As noted earlier, this completes the proof of Theorem~\ref{mainex2}.

\end{proof}

For the proof of Theorem~\ref{mainthm2}, we require one more lemma related to coding theory.

\begin{lemma}\label{codinglemma3}
For every alphabet $A$, positive integer $n$, and set $W \subset A^n$, there exists a set $S \subset W$ such that $|S| \geq \frac{|W|}{4n|A|^2}$ and $W$ is $3$-separated with respect to the Hamming distance $d$, i.e. for all 
$w, w' \in W$, $d(w,w') \geq 3$.
\end{lemma}

\begin{proof}

Choose any $A$ and $n$, and again assume without loss of generality that $A = \{0, \ldots, |A|-1\}$. Then, for any $i \in [0,2|A|)$ and $j \in [0,2|A|n)$, define $S_{n,i,j} = \{w = w_1 \ldots w_n \in W \ : \sum_{k=1}^n w_k = i \pmod{2|A|}, \sum_{k=1}^n k w_k = j \pmod{2|A|n}\}$. We claim that each $S_{n,i,j}$ is $3$-separated with respect to the Hamming distance. It is obvious that changing a single letter of a word in $S_{n,i,j}$ cannot yield another word in $S_{n,i,j}$ since changing a single letter must change the sum $\sum_{k=1}^n w_k \pmod{2|A|}$. Suppose for a contradiction that there exist $w \neq w' \in S_{n,i,j}$ differing on exactly two letters. Then $\sum_{k=1}^n w_k$ and $\sum_{k=1}^n w'_k$ are both equal to $i \pmod{2|A|}$, and differ by at most $2(|A|-1)$, and so must be equal. Similarly, $\sum_{k=1}^n k w_k$ and $\sum_{k=1}^n k w'_k$ are both equal to $j \pmod{2|A|n}$, and differ by at most $n(|A|-1) + (n-1)(|A| - 1)$, and are therefore also equal. But recall that $w$ and $w'$ differ on exactly two letters, say those indexed by $s$ and $t$. Then, $w_s + w_t = w'_s + w'_t$ and 
$sw_s + tw_t = sw'_s + tw'_t$, which implies that $w_s = w'_s$ and $w_t = w'_t$, a contradiction. We have then shown that each $S_{n,i,j}$ is $3$-separated.

We now note that the sets $S_{n,i,j}$ clearly partition $W$, and that there are $4n|A|^2$ choices for the pair $i,j$. Therefore, one of the $S_{n,i,j}$ has cardinality at least $\frac{|W|}{4n|A|^2}$; define $S$ to be that set.

\end{proof}

\begin{proof}[Proof of Theorem~\ref{mainthm2}]

Suppose for a contradiction that $X$ is a subshift with almost specification with gap function $g(n) = 1$ and two measures of maximal entropy $\mu, \nu$ with disjoint supports. For every $n$, as we did in Theorem~\ref{mainthm}, again define $\mathcal{M}_n(X) = \{w \in \mathcal{L}_n(X) \ : \ \mu([w]) > 0\}$ and $\mathcal{N}_n(X) = \{w \in \mathcal{L}_n(X) \ : \ \nu([w]) > 0\}$. We again note that by Lemma~\ref{supportcount}, $|\mathcal{M}_n(X)| \geq e^{n h(\mu)} = e^{n h(X)}$ and $|\mathcal{N}_n(X)| \geq e^{n h(\mu)} = e^{n h(X)}$ for all $n$, and that there exists $N$ so that $\mathcal{M}_n(X) \cap \mathcal{N}_n(X) = \varnothing$ for all $n \geq N$. 

For every $n$, we use Lemma~\ref{codinglemma3} to define sets $\mathcal{M}'_n \subseteq \mathcal{M}_n(X)$ and $\mathcal{N}'_n \subseteq \mathcal{N}_n(X)$ which are $3$-separated in the Hamming distance and for which $|\mathcal{M}'_n|, |\mathcal{N}'_n| \geq \frac{e^{n h(X)}}{4n|A|^2}$. We make the notation $S_n := \min(|\mathcal{M}'_n|, |\mathcal{N}'_n|)$. We now proceed somewhat as in the proof of Theorem~\ref{mainthm}, in that we will make many words in $\mathcal{L}_n(X)$ by using almost specification to nearly concatenate words in $\mathcal{M}'_j$ and $\mathcal{N}'_j$ for various $j < n$. The main difference is that rather than concatenating only two words, we now will need arbitrarily many. First, we choose $t$ such that
\begin{equation}\label{morethan1}
\sum_{j=1}^t \frac{1}{12iN|A|^2} > 1,
\end{equation}
and we denote this sum by $\alpha$. 

Now, we choose any $n > 3tN$ and create words in $\mathcal{L}_n(X)$ in the following way: define $k = \lfloor n/3tN \rfloor$, and define any 
$n_i \in [1,t]$ for $1 \leq i \leq k$. We note that $k > n/6tN$ since $n > 3tN$. Then, choose any words $w_1 \in \mathcal{M}'_{3Nn_1}$, $w_2 \in \mathcal{N}'_{3Nn_2}$, and so on, alternating between the sets, until finishing with $w_k$ in either $\mathcal{M}'_{3Nn_k}$ or $\mathcal{N}'_{3Nn_k}$, depending on whether $k$ is odd or even, respectively. Finally, choose $w_{k+1}$ in either $\mathcal{N}'_{n - \sum_{i=1}^k 3Nn_i}$ or $\mathcal{M}'_{n - \sum_{i=1}^k 3Nn_i}$, whichever is the opposite of what was chosen for $w_k$. For whichever words were chosen, use the assumed almost specification of $X$ with $g(n) = 1$ to make a word $f(w_1, \ldots, w_{k+1}) = v_1 v_2 \ldots v_{k+1} \in \mathcal{L}_n(X)$, where each $v_i$ differs from $w_i$ on at most one letter. 

We claim that this operation is injective, i.e. $f(w_1, \ldots, w_{k+1}) \neq f(w'_1, \ldots, w'_{k+1})$ unless $w_i = w'_i$ for
$1 \leq i \leq k$. Assume for a contradiction that $k$-tuples $(w_i)$ and $(w'_i)$ exist for which $f(w_1, \ldots, w_{k+1}) = f(w'_1, \ldots, w'_{k+1})$. There are two cases. 

The first case is where $n_i = n'_i$ for all $i$. Then, there exists $j$ so that $w_j \neq w'_j$. Also, since $f(w_1, \ldots, w_{k+1}) = f(w'_1, \ldots, w'_{k+1})$, both $w_j$ and $w'_j$ become the same word $v$ with at most one changed letter in each. Since $v$ and $w_j$ differ on at most one letter and $v$ and $w'_j$ differ on at most one letter, $w_j$ and $w'_j$ differ on at most two letters. This contradicts the $3$-separated property of $\mathcal{M}'_{n_j}$ and $\mathcal{N}'_{n_j}$. 

The second case is where $n_j \neq n'_j$ for some $j$. Choose $j$ minimal so that $n_j \neq n'_j$, and assume without loss of generality that $n_j < n'_j$. Then since all $w_i$ and $w'_i$ have lengths which are multiples of $3N$, the subwords $u$ and $u'$ of $w_1 w_2 \ldots w_{k+1}$ and $w'_1 w'_2 \ldots w'_{k+1}$ respectively of length $3N$ beginning at index $\sum_{i=1}^j 3N n_i + 1$ are subwords of $w_{j+1}$ and $w'_j$ respectively. Then either $u \in \mathcal{M}'_{n_{j+1}}$ and $u' \in \mathcal{N}'_{n'_j}$ or $u \in \mathcal{N}'_{n_{j+1}}$ and $u' \in \mathcal{M}'_{n'_j}$. Also, since $f(w_1, \ldots, w_{k+1}) = f(w'_1, \ldots, w'_{k+1})$, $u$ and $u'$ become the same word $u''$ after making at most one change to each. Since $|u''| = 3N$, $u''$ must contain a subword of length at least $N$ which was unchanged in both $u$ and $u'$, and therefore is a subword of each, a contradiction since no word in $\bigcup_i \mathcal{M}_i(X)$ can share an $N$-letter subword with a word in $\bigcup_i \mathcal{N}_i(X)$.

We have shown that $f$ is one-to-one, and so generates at least 
\begin{equation*}
\sum_{\substack{n_1, \ldots, n_k\\ 1 \leq n_i \leq t}} \left(\prod_{i=1}^k S_{3Nn_i}\right) S_{n - \sum_{i=1}^k 3Nn_i}
\end{equation*}
words in $\mathcal{L}_{n}(X)$. Then, by Lemma~\ref{codinglemma3},
\begin{multline*}
\sum_{\substack{n_1, \ldots, n_k\\ 1 \leq n_i \leq t}} \left(\prod_{i=1}^k S_{3Nn_i}\right) S_{n - \sum_{i=1}^k 3Nn_i} \geq \\
\sum_{\substack{n_1, \ldots, n_k\\ 1 \leq n_i \leq t}} \left(\prod_{i=1}^k \frac{e^{3N n_i h(X)}}{12N|A|^2 n_i} \right) 
\frac{e^{(n - \sum 3Nn_i) h(X)}}{4|A|^2 (n - \sum 3Nn_i )}  
\geq \frac{e^{nh(X)}}{4n|A|^2} \sum_{\substack{n_1, \ldots, n_k\\ 1 \leq n_i \leq t}} \prod_{i=1}^k \frac{1}{12N|A|^2 n_i} \\
= \frac{e^{nh(X)}}{4n|A|^2} \left(\sum_{j=1}^t \frac{1}{12N|A|^2 j}\right)^k \geq \frac{e^{nh(X)}}{4n|A|^2} \alpha^{n/6tN}.
\end{multline*}
Therefore, $\mathcal{L}_n(X) \geq (e^{nh(X)}/(4n|A|^2)) \alpha^{n/6tN}$ for all $n > 3tN$. However, taking logarithms, dividing by $n$, and letting $n \rightarrow \infty$ would imply that $h(X) \geq h(X) + \frac{1}{6tN} \ln \alpha$, a contradiction since $\alpha > 1$ by (\ref{morethan1}). Therefore, our original assumption was false, and
measures of maximal entropy $\mu$ and $\nu$ on $X$ with disjoint supports cannot exist, completing the proof.

\end{proof}

\bibliographystyle{plain}
\bibliography{weakspec}

\end{document}